\documentclass[12pt]{article}

\usepackage{a4wide}
\usepackage{times,amsmath,amsfonts,amssymb}
\usepackage{graphicx}
\usepackage{verbatim,enumerate,subfig}

\usepackage{color}
\newcommand{\old}[1]{}

\newtheorem{theorem}{Theorem}[section]

\newtheorem{proposition}[theorem]{Proposition}
\newtheorem{lemma}[theorem]{Lemma}
\newtheorem{corollary}[theorem]{Corollary}

\newenvironment{proof}{\noindent {\it Proof.~~}\ }{\rule{1mm}{2mm}\medskip}

\DeclareMathOperator{\td}{td}
\DeclareMathOperator{\tw}{tw}
\DeclareMathOperator{\diam}{diam}

\DeclareMathOperator{\w}{w}
\DeclareMathOperator{\E}{\mathbb{E}}
\DeclareMathOperator{\Var}{\mathbf{Var}}

\begin{document}



\title{On the tree--depth of random graphs
\footnote{
Work partially supported by the Catalan Research Council under
grant 2009SGR01387 and the Spanish Council under project
MTM2008-06620-C03-01. The second author wants to thank the FPU grant from the
Ministerio de Educaci\'on de Espa\~{n}a.}}


\author{G. Perarnau and O.Serra}

\maketitle
\begin{abstract}
The tree--depth is a parameter introduced under several names as a measure of sparsity of a
graph. We compute asymptotic values of the tree--depth of random graphs.
For dense graphs, $p\gg n^{-1}$, the tree--depth of a random graph $G$ is a.a.s. $\td
(G)=n-O(\sqrt{n/p})$.
Random graphs with $p=c/n$, have a.a.s. linear tree--depth when $c>1$, the tree--depth is $\Theta
(\log{n})$ when $c=1$ and
$\Theta (\log\log n)$ for $c<1$. The result for $c>1$ is derived from the computation of tree--width
and provides a more direct proof of a conjecture by Gao on the linearity of tree--width recently
proved by Lee, Lee and Oum~\cite{llo1}. We also show that, for $c=1$, every width parameter is
a.a.s. constant, and that random regular graphs have linear tree--depth.
\end{abstract}
\section{Introduction}\label{sec:intro}

An elimination tree of a graph $G$ is a rooted tree on the set of vertices such that there are no
edges in $G$ between vertices in different branches of the tree. The natural elimination scheme
provided by this tree is used in many graph algorithmic
problems where two non adjacent subsets of vertices can be managed independently. One good example
is the Cholesky decomposition of symmetric matrices (see~\cite{s1982,d1989,l1990,pt2004}).
Given an elimination tree, a distributed algorithm can be designed which
takes care of disjoint subsets of vertices in different parallel processors.
Starting by the furthest level from the root, it proceeds by
exposing at each step the vertices at a given depth. Then the algorithm runs using the
information coming from its children subtrees, which have been computed in previous steps. Observe
that the The vertices treated in different processors are independent and thus, there is no need to
share information among them.
Depending on the graph, this distributed algorithm can be more convenient that the sequential one.
In fact, its complexity is given by the height of
the
elimination tree used by the algorithm. Therefore, it is interesting to study the minimum
height of an elimination tree of $G$. This natural parameter has been introduced under numerous
names in the literature: rank function~\cite{ns1}, vertex
ranking number (or ordered coloring)~\cite{dkkm1}, weak coloring number~\cite{ky1}, but its study
was systematically undertaken by 
Ne{\v{s}}et{\v{r}}il and Ossona de Mendez under the name of \emph{tree--depth}.

%
%

The tree--depth $\td(G)$ of a graph $G$ is a measure introduced by
Ne{\v{s}}et{\v{r}}il and Ossona de Mendez~\cite{no1} in the context
of bounded expansion classes. The notion of the tree--depth is closely connected to
the tree--width. The tree--width of a graph can be seen as a measure of closeness to a tree,
while the tree--depth takes also into account the
diameter of the tree (see Section~\ref{sec:treedepth} for a precise definition and~\cite{no2011}
for an extensive account of the meaning and applications of this parameter.)


Randomly generated graphs have been widely used as benchmarks for testing distributed algorithms and
therefore it is useful to characterize the elimination tree of such graphs. The main goal of this
paper is to give asymptotically tight values for the minimum height of an elimination tree of a
random graph. 

We consider the Erd{\H{o}}s-R{\'e}nyi model $\mathcal{G}(n,p)$ for random
graph. A \emph{random graph} $G\in \mathcal{G}(n,p)$ has $n$ vertices and
every pair of vertices is chosen independently to be an edge with
probability $p$.

For any graph property $\mathcal{P}$, we say that $\mathcal{P}$ holds {\it asymptotically almost
sure (a.a.s.)}
in $G\in\mathcal{G}(n,p)$, if
\begin{equation*}
 \displaystyle\lim_{n\rightarrow \infty} \Pr(G \mbox{ satisfies } \mathcal{P})= 1
\end{equation*}

Throughout the paper, all the results and statements concerning
random graphs must be understood in the asymptotically almost sure
sense. We will occasionally make use of the $\mathcal{G}(n,m)$ model of random
graphs, where a labeled graph with $n$ vertices and $m$ edges is chosen with
the uniform distribution. As it is well--known, the two models are
closely connected and most of the statements can be transferred from
one model to the other one.

Our first result gives the value of tree--depth for dense random
graphs.
\begin{theorem}
\label{the:dense}
 Let $G\in \mathcal{G}(n,p)$ be a random graph with $p\gg n^{-1}$, then $G$ satisfies
a.a.s.
\begin{equation*}
 \td(G)=n-O\left(\sqrt{\frac np}\right).
\end{equation*}
\end{theorem}
Theorem~\ref{the:dense} implies that random graphs with a superlinear number of edges have a
tree structure
similar to the one of the complete graph. Actually our proof of Theorem~\ref{the:dense} provides
the same
result for tree--width. To our knowledge, the tree--width of a dense random graph has not been
studied before.

Ne{\v{s}}et{\v{r}}il and Ossona de Mendez showed that a sparse random graph $G(n,c/n)$ belongs a.a.s. to a bounded expansion class (see~\cite[Theorem 13.4]{no2011}). Our main result is the computation of the tree--depth of sparse random graphs.

\begin{theorem}
\label{the:sparse}
 Let $G\in \mathcal{G}(n,p)$ be a random graph with $p=\frac{c}{n}$, with $c>0$,
\begin{enumerate}
 \item[$(1)$] if $c<1$, then a.a.s. $\td(G)=\Theta (\log\log n)$
 \item[$(2)$] if $c=1$, then a.a.s. $\td(G)=\Theta (\log n)$
 \item[$(3)$] if $c>1$, then a.a.s. $\td(G)=\Theta (n)$
\end{enumerate}
\end{theorem}

The last part of this theorem is closely related to a conjecture of Gao
announced in~\cite{g3} on the linear behaviour of tree--width for
random graphs with $c=2$, inspired by the results of Kloks in~\cite{k1}. This conjecture has been
recently proved by Lee, Lee and Oum~\cite{llo1}. They show that the tree--width is linear
for any $c>1$ as a corollary of their result on the rankwidth of random graphs. Here we give a proof
of
Theorem~\ref{the:sparse}$.(3)$ which also provides a proof of Gao
conjecture, giving an explicit lower bound on the tree--width. Our
proof uses, as the one in~\cite{llo1}, the same deep result of
Benjamini, Kozma and Wormald~\cite{bkw1} on the existence of a linear order
expander in a sparse random graph for $c>1$.

The paper is organized as follows. In Section~\ref{sec:treedepth}, we
define the notion of tree--depth and give some useful results
concerning this parameter. Section~\ref{sec:dense} contains the
proof of Theorem~\ref{the:dense}, which uses the relation connecting
tree--width with balanced partitions. Finally
Theorem~\ref{the:sparse} will be proved in Section~\ref{sec:sparse}.
For $c<1$ the result follows from the fact that the random graph is
a collection of trees and unicyclic graphs of logarithmic order,
which gives the upper bound, and that there is one of these components
 with large diameter with respect to its order, providing the lower bound.
 For $c=1$ we show that the giant component in the random graph has just a
constant number of additional edges exceeding the order of a tree,
which gives the upper bound, and rely on a result of Nachmias and
Peres~\cite{np1} on the concentration of the diameter of the giant
component to obtain the lower bound. Finally, as we have already
mentioned, for $c>1$ the result follows readily from the existence
of an expander of linear order in a sparse random graph for $c>1$, a
fact proved in Benjamini, Kozma and Wormald~\cite{bkw1}.

\section{Tree--depth}\label{sec:treedepth}

Let $T$ be a rooted tree. The \emph{closure} of $T$ is the graph
that has the same set of vertices and an edge between every pair of
vertices such that one is an ancestor of the other in the rooted
tree. A \emph{rooted forest} is a disjoint union of rooted trees. 
 The height of a rooted forest is the maximum height of its trees.
The closure of a rooted forest is the disjoint union of the closures of its rooted trees.

The tree (forest) $T$ is called an elimination tree (forest) for $G$ if $G$ is a subgraph of its
closure.
The \emph{tree--depth} of a graph $G$ is defined to be the
minimum height of an elimination forest of $G$.
Some examples are shown in Figure~\ref{fig:treedepth}.
\begin{figure}[ht]
 \begin{center}
 \includegraphics[width=0.9\textwidth]{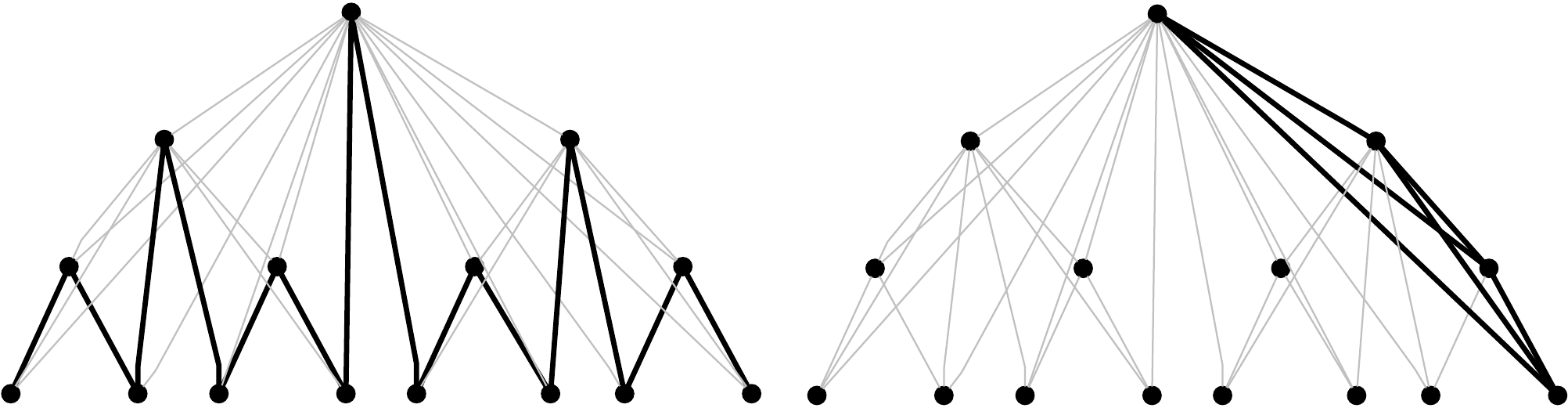}
 \end{center}
 \caption{The path with $15$ vertices and the complete graph $K_4$ have tree--depth $4$}
 \label{fig:treedepth}
\end{figure}

If $G$ has connected components
$H_1,\dots,H_s$, then,
\begin{equation}
 \label{eq:cc}
 \td(G)= \displaystyle\max_{0<i\leq s}\td(H_i)
\end{equation}
by the definition of the height of a forest.

It is clear from the definition that the tree--depth of a graph $G$ with $n$ vertices satisfies $\td
(G)\le n$, and that the equality is satisfied only for the complete graph $K_n$. Note that the
following inequality holds,
\begin{equation}
\label{eq:decreasing}
 \td(G\backslash v)\geq \td(G) -1.
\end{equation}

For any tree $T$, it can be checked by induction that,
\begin{equation}
 \label{eq:trees}
 \td(T)\leq \lfloor \log n\rfloor+1,
\end{equation}
and that the equality holds if and only if $T$ is a path $P_n$ with $n$ vertices,
 \begin{equation}
\label{eq:paths}
 \td(P_n)= \lfloor\log n\rfloor+1.
 \end{equation}
Henceforth by $\log$ we denote the logarithm in base two while we will use $\ln$ for the natural
logarithm.

As a simple consequence of the previous equation and the fact that the tree--depth is
monotonically increasing under the subgraph ordering, we have that if $G$ contains a path of length
$t$, then $\td(G)\geq \log{t}$. In particular, if $G$ has diameter $d$,
 \begin{equation}
\label{eq:diam}
 \td(G)\geq \log{d}
 \end{equation}

The following inequalities relate the tree--width and tree--depth of
a graph (see e.g.~\cite{bghk95})
\begin{equation}
\label{eq:tdtw}
 \tw(G)\leq \td(G) \leq \tw(G)(\log n+1).
\end{equation}
 Note that there are classes of graphs that have bounded tree--width but
unbounded tree--depth, for example trees. On the other hand, if a class of graphs has bounded
tree--depth, then it also has bounded tree--width.

Some of our results use the relation~\eqref{eq:tdtw} between tree--depth and tree--with.
Bounds on tree--width can be obtained through its connection with balanced separators.
Kloks~\cite{k1} defines a partition $(A,S,B)$ of a graph $G$ with $n$ vertices to be a
{\it balanced $k$--partition} if $|S|=k+1$ and
\begin{equation}\label{eq:ksep}
\frac{1}{3}\left( n-k-1\right) \leq |A|,|B| \leq \frac{2}{3}\left( n-k-1\right).
\end{equation}
 He states the following result connecting balanced partitions and tree--width.

\begin{lemma}[Kloks~\cite{k1}]
 \label{lem:balanced}
Let $G$ be a graph with $n$ vertices and $\tw(G)\leq k$,
$k\leq n-4$. Then $G$ has a balanced $k$-partition.
\end{lemma}

 We will often use the non existence of a balanced partition to provide
lower bounds on the
tree--width of $G$.



\section{Tree--depth of dense random graphs}\label{sec:dense}

This section is devoted to prove Theorem~\ref{the:dense}.
\vspace{0.5cm}

{\scshape Proof of Theorem~\ref{the:dense}. }
We only need to prove the lower bound. It will be derived from the analogous one for
tree--width through inequality~\eqref{eq:tdtw}. For this we will show that a random graph
$G\in\mathcal{G}(n,p)$ with $p\gg n^{-1}$ does not contain a balanced separator of order less than
$n-3\sqrt{\tfrac{\ln{3}}{2}}\sqrt{\tfrac{n}{p}}$. The result will follow
from Lemma~\ref{lem:balanced}.

%

 Assume $p=c(n)/n$ and set any function $f(c)> 3\sqrt{\tfrac{\ln 3}{2c}}$.
Suppose that there exist a balanced $k$--partition $(A,S,B)$ of $G$ with 
$k\leq (1-f(c)) n$. By definition of
balanced $k$--partitions, we have
$|S|=k+1, \mbox{ and } |A|,|B|\geq\tfrac{f(c)n}{3}$, so that
$$
|A|\cdot|B|\ge \frac{2f(c)^2}{9}n^2.
$$
Let $X_{(A,S,B)}$ denote the event that $(A,S,B)$ is a balanced $k$--partition of $G$ with
$k\leq (1-f(c)) n$.
 We have
\begin{equation}
 \Pr\left(X_{(A,S,B)}\right)=\left(1-p\right)^{|A||B|} \leq
 \left(1-p\right)^{\frac{2f(c)^2}{9}n^2}.
\end{equation}

Let $\mathcal{B}$ denote the set of balanced $k$--partitions with $k\le (1-f(c)) n$. By using the
trivial bound $|\mathcal{B}|\le 3^n$, the
number of labeled partitions of $[n]$ into three sets, we get:

\begin{eqnarray}
\label{eq:bound}
\Pr\left(\cup_{(A,S,B)\in \mathcal{B}} X_{(A,S,B)}\right)
&\leq& \displaystyle\sum_{(A,S,B)\in \mathcal{B}} \Pr\left( X_{(A,S,B)}\right)\nonumber \\
&\leq& 3^n\left(1-p\right)^{\frac{2f(c)^2}{9}n^2}\nonumber\\
&\leq& \exp\left\{(\ln{3})n-\frac{2f(c)^2}{9}pn^2\right\}\\
&\rightarrow& 0\;\; (n\to \infty)\nonumber,
\end{eqnarray}
where we have used that $1-x\leq e^{-x}$ and $f(c)>3 \sqrt{\tfrac{\ln 3}{2c}}$.


Since there is no set of size $(1-f(c))n$ separating $G$, from inequality \eqref{eq:tdtw} and 
Lemma~\ref{lem:balanced} we have $\td(G)\geq\tw(G)> (1-f(c)) n$.
The above inequality is valid for all $f(c)> 3\sqrt{\frac{\ln 3}{2c}} $ and thus, we have
$\td(G)\geq n-O(\sqrt{n/p})$.

\vspace{0.3cm}

Observe that the above argument can be used to deduce that the tree--width of
sparse random graphs with $p(n)=c/n$ is linear in $n$ for sufficiently large $c$ and
obtain a lower bound for the constant. From~\eqref{eq:bound} we need
$$
f(c)>\sqrt{\frac{9\ln{3}}{2c}}
$$
and since the treewidth is $(1-f(c))n$, we also need $f(c)<1$. This two conditions imply that for
any $c>\frac{9\ln{3}}{2}\approx 4.94$, one has that
\begin{equation*}
 \td(\mathcal{G}(n,c/n))\geq \tw(\mathcal{G}(n,c/n))\geq (1-f(c))n = \Omega(n).
\end{equation*}
In the next section we will prove that a.a.s. $\td(\mathcal{G}(n,c/n))$ is linear for any $c>1$.


\section{Tree--depth of sparse random graphs}\label{sec:sparse}

In this section Theorem~\ref{the:sparse} will be proved. The proof is divided in three cases
depending on the value of $c$.

\subsection{Proof of Theorem~\ref{the:sparse}.1}
\label{ssc:c<1}

Let $G\in \mathcal{G}(n,p=c/n)$ with $0<c<1$. Our goal is to show that
$\td(G)=\Theta(\log\log n)$.

 We will first prove the upper bound.
A \emph{unicyclic} graph (\emph{unicycle}) is a connected graph that has the same number of
vertices than edges, that is, the graph contains exactly one cycle.

\begin{lemma}
\label{lem:pseudo}
 If each connected component of $G$ is either a tree or a unicycle, then $\td(G)\leq
\log n_c + 2$, where
$n_c$ is the cardinality of the largest connected component of $G$.
\end{lemma}

\begin{proof} As it has been remarked, $\td (G)$ equals the tree--depth of the largest connected
component of the graph. After deleting at most one vertex, a connected component becomes a tree. The
result follows by using~\eqref{eq:decreasing} and ~\eqref{eq:trees}.

\end{proof}

One of the central results of Erd{\H{o}}s and R{\'e}nyi~\cite{er1} states that,
if $0<c<1$, then $G$ is composed by trees and unicycles.
Moreover the order of the largest component in the random graph is $\Theta(\log n)$ (see
e.g.~\cite{b1}).
Therefore, $n_c =\Theta (\log n)$ and, by
Lemma~\ref{lem:pseudo}, we have
\begin{equation*}
 \td(G) = O(\log\log n).
\end{equation*}

We next show the lower bound. 
Recall that the diameter of the graph provides a lower bound on the tree--depth of $G$ by
Eq.~\eqref{eq:diam}.
Hence our proof for this case will be completed if we show that the random graph $G$ a.a.s. contains
a tree $T$ of order $\Theta (\log n)$ and sufficiently large diameter.

Observe that every labeled tree on $k$ vertices has the same probability to appear in $G$ as a
connected component, since each tree with fixed order contains the same number of edges. R{\'e}nyi
and Szekeres~\cite{rs1} proved that the expected height $H_k$
 of a random labeled tree on $k$ vertices satisfies
\begin{equation*}
 \mathbb{E}(H_k)\sim \sqrt{2\pi k} \quad \mbox{and}
 \quad
\mathbf{Var}(H_k)\sim \frac{\pi(\pi-3)}{3}k.
\end{equation*}
Hence the diameter $D_k$ of a labelled tree with $k$ vertices satisfies $H_k\leq D_k\leq 2H_k$ and
an
thus, $\mathbb{E}(D_k)=\Theta(\sqrt{k})$. Unfortunately we have that $\Var(D_k)=\Theta(k)$ and we
can not show that the random variable $D_k$ is highly concentrated around its expected value. The
diameter of an individual tree can be even constant as $k\to \infty$. However, the number of
trees with $k$ vertices in $G$ is large enough to ensure that a.a.s. there is at least one with
sufficiently large diameter.

To count the number of trees of each size it is better to use the random model $\mathcal{G}(n,m(n))$
where a graph is chosen uniformly at random from all the labeled graphs with $m(n)$ edges.
Erd{\H{o}}s and R{\'e}nyi~\cite{er1}
proved that if $m(n)= \Omega(n)$, the random variable $X_k$ counting the number of trees of order
$k$ in $\mathcal{G}(n,m(n))$ follows a normal distribution with expected value and variance $M_k$,
where
\begin{equation*}
 M_k = n \frac{k^{k-2}}{k!}\left( \frac{2m}{n}\right) ^{k-1} e^{-\frac{2km}{n}}.
\end{equation*}

Moving back to the random graph model $\mathcal{G}(n,p)$ with $p=c/n$, and
noting that $\mathbb{E}(m)=\frac{cn}{2}$, we get the analogous
result, where now
\begin{equation}
\label{eq:mn}
 M_k = n \frac{k^{k-2}}{k!} c ^{k-1} e^{-kc}.
\end{equation}

We are interested in $X=X_{\log{n}}$, the number of trees of order $\log{n}$, for which,
\begin{equation*}
 M=M_{\log{n}}= \frac{n^{\log\log n -\alpha}}{c(\log^2 n)(\log n)!}
\end{equation*}
where $\alpha= c-1-\log c$. Observe that $M\rightarrow\infty$ when $n\to
\infty$, i.e. we expect infinitely many trees of size $\log{n}$.
Chebyshev's inequality with $\E(X)=\Var(X)= M$
 ensures that a.a.s the number of tree components with order $\log n$ is $X =
(1-o(1))M \rightarrow\infty \;\; (n\to \infty)$.

Denote by $D= D_{\log{n}}$ and define $\overline{D}$ to be the mean diameter over all the components
of order $\log n$. Clearly $\E(\overline{D})= \E(D)=\Theta(\sqrt{\log n})$, but now, since the
diameter of each individual tree is independent from the other ones, the variance of $\overline{D}$
is smaller. In fact, $\Var(\overline{D})=o(\log n)$. Hence, by using Chebyshev inequality on
$\overline{D}$, we can ensure that $\overline{D}=\Theta(\sqrt{\log n})$ a.a.s. and there exists
a tree $T$ in $G$ with diameter $d=\Omega (\sqrt{\log n})$. Hence, by \eqref{eq:diam},
\begin{equation*}
 \td(G)=\Omega(\log d) =\Omega \left(\log {\left(\sqrt{\log n}\right)}\right) =\Omega (\log
\log n).
\end{equation*}

This completes the proof of this case. 


\subsection{Proof of Theorem~\ref{the:sparse}.2}\label{ssc:c=1}

Now we look at the critical point where $c=1$.

It was showed by Erd\H{o}s and R\'enyi also in~\cite{er1} that $p=1/n$ is the threshold
probability for the existence of a polynomial size component, the so called \emph{giant component}
of the random graph. For this particular probability, the largest component has order $O(n^{2/3})$.

We first prove the upper bound. We show that each component is similar to a
tree. For convenience we use the definition of a $(k,\ell)$-component given in~\cite{jlr1}.
A \emph{$(k,\ell)$-component}, $\ell\geq -1$, is a connected component with $k$ vertices and
$k+\ell$ edges. Thus, $(k,-1)$--components are trees and $(k,0)$--components correspond to
unicyclic graphs. A \emph{complex component} is a $(k,\ell)$--component with $\ell>0$.


\begin{proposition}\label{prop:excess}
	 All complex components of $G(n,1/n)$ have constant excess $\ell$.
\end{proposition}

\begin{proof}
Let $Y(k,\ell)$ denote the number of $(k,\ell)$-components of $G(n,1/n)$. The expected value of
$Y(k,\ell)$ is

$$
\mathbb{E} (Y(k,\ell)) = \binom{n}{k} C(k,\ell) \left(\frac{1}{n}\right)^{k+\ell}
\left(1-\frac{1}{n}\right)^{\binom{k}{2}-(k+\ell) +k(n-k)} \leq \frac{C(k,\ell)}{n^{\ell}
k!}e^{-k(1+o(1))},
$$
where $C(k,\ell)$ is the number of connected labeled graphs with $k$ vertices and $k+\ell$ edges and
$k/n\to 0$.
B{\'o}llobas~\cite[Corollary 5.21]{b1} obtained the following sharp bound for $C(k,\ell)$,
$$
C(k,\ell)\leq O\left(\ell^{-\ell/2} k^{k+(3\ell-1)/2}\right).
$$
By combining the previous expressions and using Stirling formula, for $k=O(n^{2/3})$ one gets
$$
\mathbb{E} (Y(k,\ell)) \leq \frac{\ell^{-\ell/2}}{k}\left(\frac{k^{3/2}}{n}\right)^{\ell} \leq
O\left(\frac{\ell ^{-\ell/2}}{n^{2/3}}\right),
$$
if $\ell \geq 2$.

Let $Y(\ell)=\sum_{k}Y(k,\ell)$ denote the total number of connected components with excess $\ell$.
Since connected components have order at most $O(n^{2/3})$, we have
$$
\mathbb{E} (Y(\ell)) = \displaystyle\sum_{k=0}^{O(n^{2/3})} \mathbb{E} (Y(k,\ell)) \leq
 O\left(\ell ^{-\ell/2}\right).
$$
Let $Y_K=\sum_{\ell\ge K^2} Y(\ell)$ denote the total number of complex components with excess at
least $K^2$. We have
$$
\mathbb{E}(Y_K)=\sum_{\ell\ge K^2}
\mathbb{E} (Y(\ell))\leq \sum_{\ell\ge K^2}O\left( \ell^{-\ell/2}\right) \leq
\sum_{\ell\ge K^2} O\left(\frac{1}{K^\ell}\right)=O\left(\frac{1}{(K-1)K^{K^2-1}}\right),
$$
Using the Markov inequality we have that

$$
\Pr(Y_K\geq 1)\leq \E(Y_K)\to 0 \;\;(K\to \infty).
$$

This implies that a.a.s. there are no $\ell$-components with $\ell=\omega(1)$, i.e. the excess of
all the components is constant when $p=1/n$.

\end{proof}

Each component $C$ contains
$k=O(n^{2/3})$ vertices and, by Proposition~\ref{prop:excess}, $k+\ell(C)$ edges with
$\ell(C)=O(1)$. Note that we can delete $\ell(C)$
vertices, turning $C$ into a tree of order
$(k-\ell(C))=O(n^{2/3})$. Let $\ell_m$ be the maximum excess of edges among all the components. The
tree--depth satisfies
\begin{equation*}
  \td(G)\leq \ell_m + O\left(\log \left(n^{2/3}\right)\right)= O(1) + O\left(\frac{2}{3}\log
n\right)= O(\log n),
\end{equation*}
giving the upper bound.
\vspace{0.3cm}

We note that the same argument can be applied to give a constant upper bound for the
tree--width of $G(n,1/n)$,
\begin{proposition}\label{prop:width}
 Let $G\in \mathcal{G}(n,p=1/n)$. Then,
\begin{equation*}
 \tw(G) = O(1).
\end{equation*}
\end{proposition}
\begin{proof}
The giant component of $G$ have order $O(n^{2/3})$. By Proposition~\ref{prop:excess}, all the
components have constant excess $\ell$. Note that adding one edge to a graph increases its
tree--width by at most one unit and that the tree--width of a tree is one.
Hence, denoting by $\ell_m$ the maximum excess of components in $G$, we have
$
\tw (G)\le 1+\ell_m (G)=O(1).
$

\end{proof}

Observe that the previous proposition applies to any width parameter $\w(G)$ which can be
upper bounded by a
function of the tree--width. Examples of such parameters are branch--width, path--width,
rank--width or clique--width.

 Note also that for this probability, $p=1/n$, the latter inequality of~\eqref{eq:tdtw}
is asymptotically
tight.

The lower bound is obtained by an argument on the diameter of a giant component.
Each giant component $C$ is a tree decorated with a constant number $\ell(C)$ of extra
edges (see e.g.~\cite{b1}). Observe that adding an edge to a connected component can at most
halve its
diameter. Since the expected diameter of a random tree of size $k$
is $\Theta(\sqrt{k})$, it follows from Proposition~\ref{prop:excess} that the expected diameter of a
giant component is $\Theta(n^{1/3})$. Here, however, there are too few giant components, and we can
not use the same argument as for the previous
case $c<1$ to conclude that the graph contains a giant component with the expected
value of the diameter. The concentration of this random variable follows from a
more general statement due to Nachmias and Peres~\cite{np1} on the
diameter of the largest component of a random graph with $p=1/n$.

\begin{theorem}[\cite{np1}]\label{thm:np} Let $C$ be the largest component of a
random graph in $\mathcal{G}(n,1/n)$. Then, for any
$\varepsilon>0$, there exists $A=A(\varepsilon)$ such that
  \begin{equation*}
    \Pr(\diam(C)\notin (A^{-1}n^{1/3}, An^{1/3}) ) <
    \varepsilon.
  \end{equation*}
\end{theorem}

By using the Theorem~\ref{thm:np} and \eqref{eq:diam}, we get $$\td(G)=
\Omega(\log n^{1/3}) =
\Omega(\log n).$$ This concludes the proof of the case $c=1$.


\subsection{Proof of Theorem~\ref{the:sparse}.3}\label{ssc:c>1}

On the range of $p=c/n$ for $c>1$, we will prove that the tree--width is linear a.a.s., implying
the same result for the tree--depth.

We have already seen at the end of Section~\ref{sec:dense} that the tree--width of
$\mathcal{G}(n,c/n)$ is linear for any $c>4.94$.

Kloks~\cite{k1} studied the tree--width of a random graph $G$ with $p=c/n$ and proved that it is
linear for any $c>2.36$. Gao~\cite{g3} showed that the lower bound can be improved to 
$c>2.162$, and
conjectured that the threshold for having linear tree--width occurs for some $1<c<2$.
As a side result of their study on the rankwidth of random graphs, Lee, Lee and
Oum~\cite{llo1} settled the conjecture by showing that the
tree--width is linear for any $c>1$.

Here we will give a more direct proof of the linearity of tree--width for random graphs with $c>1$
 which provides an explicit lower bound for the linear constant. Our approach uses the same deep
result of Benjamini, Kozma and Wormald, Theorem~\ref{thm:expander} below,
as in~\cite{llo1}.

Recall that the (edgewise) Cheeger constant of a graph $G$ is defined as

\begin{equation}\label{eq:cheeger}
 \Phi(G) =\displaystyle\min_{\substack{ X\subseteq V\\ 0<|X|\leq n/2 }} \frac{e(X,V\backslash
X)}{e(X,V)}
\end{equation}
where $e(X,Y)=\sum_{x\in X} \deg_Y(x)$.

For any fixed $\alpha>0$, $G$ is an {\em $\alpha$--edge-expander} if
$\Phi(G)>\alpha$.



The recent proof of Benjamini, Kozma and Wormald~\cite{bkw1} for the value of the
mixing time of a random walk on the giant component of a random
graph with $p=c/n$, $c>1$, relies on the existence of an
$\alpha$--edge-expander connected subgraph of linear size in the giant component.
Theorem~\ref{thm:expander} below is a direct consequence of
~\cite[Theorem 4.2]{bkw1}, which ensures the existence a.a.s. of a certain
subgraph $R_N(G)$  of the giant component which is an $\alpha$--edge-expander for
some sufficiently small $\alpha$ (which has some additional properties.)
The fact that this subgraph has linear order arises in the proof of
this theorem (see~\cite[page 19]{bkw1}.) We note that in the above mentioned
paper the authors use a different formulation of the Cheeger constant which can be easily
shown to be equivalent to \eqref{eq:cheeger}; see e.g.~\cite{llo1}.

\begin{theorem}[\cite{bkw1}]\label{thm:expander} Let $G$ be a random graph in $\mathcal{G}(n,p)$
with $p=c/n$, $c>1$. There exist $\alpha, \delta >0$ and a subgraph
$H\subseteq G$ such that $H$ is an $\alpha$--edge-expander and $|V(H)|=\delta
n$.
\end{theorem}


%
%


Let $H$ be the subgraph obtained in Theorem~\ref{thm:expander}. By Theorem~\ref{lem:balanced}, there
exists a balanced partition $(A,S,B)$ in $H$, where $S$ is a vertex separator with
$|S|=\tw(H)+1$ and
$|A|\geq (\delta n-|S|)/3$.

Since $H$ is connected we have $e(A,V(H))\geq |A|$. Then,
$$
|A|\leq e(A,V(H)) \leq \frac{e(A,S)}{\alpha},
$$
implying that $e(A,S)\geq \alpha|A|\geq \alpha\frac{\delta n-|S|}{3}=\beta n$, where
$\beta=\alpha(\delta-\gamma)/3$ and $\gamma=|S|/n$.

Since $e(A,S)\le e(S, V(G))$, it suffices to show that any set $S$ of vertices with
$e(S,V(G))\geq \beta n$, must
have linear order. For this we show that there is $\gamma_0>0$, which depends only on $\alpha,
\delta$ and $c$, such that the probability that a set $S$ with $|S|\le \gamma_0 n$ satisfies
$e(S,V(G))\ge \beta n$ tends to zero when $n\to \infty$. We use a union bound on the number of sets
$S$ of size $\gamma n$, with $\gamma <\alpha\delta/(3c+\alpha)$, and the fact that $e(S,V(G))$
is a binomial random variable $Bin( \gamma n^2, c/n)$. We have,

\begin{eqnarray*}
\Pr(\exists S: \, |S|=\gamma n,\, e(S,V(G))\geq \beta n ) &\leq& \binom{n}{\gamma n}
\sum_{e=\beta n}^{\gamma n^2} \binom{\gamma n^2}{e} p^{e}(1-p)^{\gamma n^2-e}\\
 \mbox{{\small (since $\beta n>c\gamma n=\E (e(S, V(G)))$, if $\gamma
<\alpha\delta/(3c+\alpha)$)}}
&\leq& \binom{n}{\gamma n} \gamma n^2\binom{\gamma n^2}{\beta n} p^{\beta
n}(1-p)^{\gamma n^2-\beta n}\\
&\leq& \binom{n}{\gamma n} \gamma n^2\binom{\gamma n^2}{\beta n} p^{\beta
n}\\
 \mbox{{\small(since $\binom{x}{y}\leq \left(\frac{xe}{y}\right)^y$)}}&\leq& \gamma n^2
\left(\left(\frac{e}{\gamma}\right)^{\gamma}
\left(\frac{\gamma e c}{\beta}\right)^{\beta} \right)^n.\\
\end{eqnarray*}

 Since $\beta=\alpha(\delta-\gamma)/3$, the expression
$\left(\frac{e}{\gamma}\right)^{\gamma}\left(\frac{\gamma e
c}{\beta}\right)^{\beta}$
tends to $0$ when $\gamma\to 0$. Thus, there exists $0<\gamma_0<\alpha\delta/(3c+\alpha)$ such that
$\left(\frac{e}{\gamma}\right)^{\gamma}\left(\frac{\gamma e
c}{\beta}\right)^{\beta}<1$. It follows that

$$
\Pr(\exists S: \, |S|\leq\gamma_0 n,\, e(S,V(G))\geq \beta n )\to 0 \;\; (n\to \infty).
$$

Therefore the set $S$ has size at least $\gamma_0 n$ and $\tw(H)\geq \gamma_0 n$. Since the
tree--width is monotone with respect to the subgraph relation and $H\subseteq G$, we know that
$\tw(H)\leq \tw(G)$, and $\tw(G)= \Omega(n)$, concluding the proof of Theorem~\ref{the:sparse}.

\vspace{0.3cm}

Observe that it is not true in general that every set $S$ with $f(n)$ 
incident edges has size
$\Theta(f(n))$. For instance, the maximum degree of $\mathcal{G}(n,c/n)$ ($c>1$) is not
constant; see e.g.~\cite{b1}.

Moreover, it can be checked that the above argument gives
$tw(G)\geq\frac{(\alpha\delta)^2}{9e^3c^2}n$, thus providing an explicit lower bound for $tw(G)$.
In~\cite{llo1} the authors provide the lower bound $\frac{\alpha\delta}{M^2}n$, where $M$ is
constant but it is not made explicit.



\vspace{0.3cm}

We finish the paper by showing, with analogous arguments, that the tree--width and
tree--depth of
random regular graphs is also linear. We consider the configuration model $\mathcal{G}(n,d)$
as a
model of random regular graphs (see e.g.~\cite{b1}).
\begin{corollary}
For any fixed $d\geq 3$, random $d$-regular graphs have linear tree--width and linear tree--depth.
\end{corollary}

\begin{proof}
The Cheeger constant can be bounded in terms of the second smallest eigenvalue $\mu_2(G)$ of the
Laplacian matrix (see e.g.~\cite[Lemma 2.1]{c1}),
\begin{equation*}
\Phi(G)\geq \frac{\mu_2(G)}{2}
\end{equation*}
Since $G$ is $d$--regular $\mu_2(G)=d-\lambda_2(G)$ where $\lambda_2(G)$ is the second largest
eigenvalue of the Adjacency matrix.
Friedman, Kahn and Szemer\'{e}di~\cite{fks1} proved that this
eigenvalue in $d$--regular random graphs is a.a.s. $O(\sqrt{d})$. Therefore,
\begin{equation*}
\Phi(G)\geq \frac{d-O(\sqrt{d})}{2}=\alpha(d) >0
\end{equation*}
and $G\in\mathcal{G}(n,d)$ is $\alpha(d)$--edge-expander. However, any set of vertices $X$ has a
large set of neighbors, i.e. the graph is not only edge-expander but vertex-expander. If
$N(X)=\{v\in V:\, \exists u\in X,\, u\sim v\}$, then
$$|N(X)\setminus X|\ge e(X,V\setminus X)/d\ge \Phi (G)|X|\ge
\alpha (d)|X|$$
and the graph is an $\alpha(d)$--vertex-expander.

By Theorem~\ref{lem:balanced}, we know that there is a balanced partition $(A,S,B)$ of $G$, where
$S$ is a vertex separator of cardinality $\tw(G)+1$ and $|A|\geq (n-\tw(G)-1)/3$.
Since $N(A)\setminus A= S$, we have
 \begin{equation*}
  \tw(G)= |S|-1\geq \alpha(d)|A|-1\geq \alpha(d)\frac{n-\tw(G)-1}{3}-1
 \end{equation*}
where we have used the fact that the set $A$ is vertex-expander.
Thus, $\tw(G)\geq \frac{\alpha(d)(n-1)-3}{\alpha(d)+3}= \Omega(n)$.

\end{proof}

\bibliographystyle{amsplain}
\bibliography{bibliografia}

\end{document}